\documentclass[11pt]{article} 
\usepackage{amsmath,amssymb,amsthm}
\usepackage{fancybox}  
\usepackage{graphicx}
\usepackage{color}

\allowdisplaybreaks

\begin{document}

\def\fl#1{\left\lfloor#1\right\rfloor}
\def\cl#1{\left\lceil#1\right\rceil}
\def\ang#1{\left\langle#1\right\rangle}
\def\stf#1#2{\left[#1\atop#2\right]} 
\def\sts#1#2{\left\{#1\atop#2\right\}}
\def\eul#1#2{\left\langle#1\atop#2\right\rangle}
\def\N{\mathbb N}
\def\Z{\mathbb Z}
\def\R{\mathbb R}
\def\C{\mathbb C}
\newcommand{\ctext}[1]{\raise0.2ex\hbox{\textcircled{\scriptsize{#1}}}}

\newtheorem{theorem}{Theorem}
\newtheorem{Prop}{Proposition}
\newtheorem{Cor}{Corollary}
\newtheorem{Lem}{Lemma}
\newtheorem{Def}{Definition}
\newtheorem{Conj}{Conjecture}

\newenvironment{Rem}{\begin{trivlist} \item[\hskip \labelsep{\it
Remark.}]\setlength{\parindent}{0pt}}{\end{trivlist}}

\title{$p$-numerical semigroups with $p$-symmetric properties
}

\author{
Takao Komatsu 
\\
\small Department of Mathematical Sciences, School of Science\\[-0.8ex]
\small Zhejiang Sci-Tech University\\[-0.8ex]
\small Hangzhou 310018 China\\[-0.8ex]
\small \texttt{komatsu@zstu.edu.cn}\\\\
Haotian Ying
\\ 
\small Department of Mathematical Sciences, School of Science\\[-0.8ex]
\small Zhejiang Sci-Tech University\\[-0.8ex]
\small Hangzhou 310018 China\\[-0.8ex]
\small \texttt
{tomyinght@gmail.com}  
}

\date{
\small MR Subject Classifications: 20M14, 11D07, 20M05, 05A15, 11B25 
}

\maketitle
 
\begin{abstract} 
The so-called Frobenius number in the famous linear Diophantine problem of Frobenius is the largest integer such that the linear equation $a_1 x_1+\cdots+a_k x_k=n$ ($a_1,\dots,a_k$ are given positive integers with $\gcd(a_1,\dots,a_k)=1$) does not have a non-negative integer solution $(x_1,\dots,x_k)$. The generalized Frobenius number (called the $p$-Frobenius number) is the largest integer such that this linear equation has at most $p$ solutions. That is, when $p=0$, the $0$-Frobenius number is the original Frobenius number. 

In this paper, we introduce and discuss $p$-numerical semigroups by developing a generalization of the theory of numerical semigroups based on this flow of the number of representations. That is, for a certain non-negative integer $p$, $p$-gaps, $p$-symmetric semigroups, $p$-pseudo-symmetric semigroups, and the like are defined, and their properties are obtained. When $p=0$, they correspond to the original gaps, symmetric semigroups, and pseudo-symmetric semigroups, respectively.  
\\
{\bf Keywords:} numerical semigroup, Frobenius number, genus, Ap\'ery set, pseudo-Frobenius number, symmetry    
      
\end{abstract}

\section{Introduction}  

Let $a_1,a_2,\dots,a_k$ ($k\ge 2$) be a sequence of positive integers with $\gcd(a_1,a_2,\dots,a_k)=1$. Then, there are integers that can be represented by linear combinations of non-negative integer coefficients by $a_1,a_2,\dots,a_k$, and integers that cannot be represented. Since it is easy to see that any sufficiently large integer can always be represented, there is always the largest integer that cannot be represented. This integer is called the Frobenius number, and the problem associated with this number is called the linear Diophantine problem of Frobenius. We denote it by $g(a_1,a_2,\dots,a_k)$. This problem has long been known and been captured attractively by many experts and amateur mathematicians as it is also known as the stamp exchange problem associated with daily life and can be considered as part of the more general Diophantine equations' problem.  
Currently, there are many topics on the linear Diophantine problem of Frobenius, but one of the most central problems is the explicit formula of the Frobenius number for given $a_1,a_2,\dots,a_k$. Indeed, in the case of two variables, the explicit formula of the Frobenius number was given by Sylvester at the end of the 19th century, but in the case of three or more variables, the general explicit formula has not been known. Moreover, the Frobenius number cannot be given by closed formulas of a certain type \cite{cu90}, the problem to determine the Frobenius number is NP-hard under Turing reduction (see, e.g., Ram\'irez Alfons\'in \cite{ra05}).  Nevertheless, some explicit forms have been discovered in some particular cases, including arithmetic, geometric-like, Fibonacci, Mersenne, and triangular and so on (see, e.g., \cite{RR18,RBT2015,RBT2016,RBT2017} and references therein).

As mentioned above, there are integers that can be represented by linear combinations of non-negative integer coefficients by $a_1,a_2,\dots,a_k$, and integers that cannot be represented, but for those integers can be represented, there may be two or more different representations. The number of these representations is the most natural and most rational generalization of the Frobenius number. Therefore, for a given non-negative integer $p$, we can consider the largest integer $g_p(a_1,a_2,\dots,a_k)$ such that the number of expressions that can be represented by $a_1,a_2,\dots,a_k$ is at most $p$ ways. We call it the generalized Frobenius number or $p$-Frobenius number (which is redefined from the viewpoint of numerical semigroups in the later section). That is, all integers larger than this number have the number of representations of $p+1$ or more ways. When $p=0$, this number is reduced to the original Frobenius number $g(a_1,a_2,\dots,a_k)=g_0(a_1,a_2,\dots,a_k)$. 
This generalized Frobenius number is also called the $k$-Frobenius number \cite{ADL16,BDFHKMRSS} or the $s$-Frobenius number \cite{FS11}. 

A slightly different but similarly directed definition of the generalized Frobenius number has been also studied. One can consider the largest integer $g_p^\ast(a_1,a_2,\dots,a_k)$ that has exactly $p$ distinct representations (see, e.g., \cite{BDFHKMRSS,FS11}). However, in this case, the ordering $g_0^\ast\le g_1^\ast\le \cdots$ may not hold. For example, $g_{17}^\ast(2,5,7)=43>g_{18}^\ast(2,5,7)=42$. In addition, for some $j$, $g_j^\ast$ may not exist. For example, $g_{22}^\ast(2,5,7)$ does not exist because there is no positive integer whose number of representations is exactly $22$. Therefore, we do not adopt this definition, as this shortcoming may have no problem in studies in other directions, but not in the study of explicit formulas, as well as in the study of numerical semigroups, which is the central topic of this paper.  

In \cite{BDFHKMRSS}, a generalization of the Frobenius problem was considered, where the object of interest is the greatest integer which has exactly $p$ representations by a collection of positive relatively prime integers. 
In \cite{FS11}, general upper and lower bounds were obtained on the generalized Frobenius number, that is, the largest positive integer that has precisely $p$ distinct representations.  
In \cite{ADL16}, it was proved that for fixed $k$ and $p$, the $p$-Frobenius number can be computed in polynomial time. 

It is easy to give an explicit form of $p$-Frobenius number when $k=2$. However, for $k\ge 3$ and $p\ge 1$, only bounds or algorithms on $p$-Frobenius number has been available. Even in the case of $p=0$, it is not easy to obtain the explicit formula of the Frobenius number for $k\ge 3$ even in a special case. As one can easily imagine, it becomes even more difficult when $p\ge 1$. As far as it is known, no explicit formula was known for $p>1$. Attempts to break this shell are one of the motivations to study the theory of numerical semigroups.     

Recently, in \cite{Ko-repunit,Ko-triangular}, with the help of the theory of numerical semigroups, we are finally successful to obtain explicit forms of $p$-Frobenius number of some special sequences. That is, we give the generalized Frobenius number when $a_j=(b^{n+j-1}-1)/(b-1)$ ($b\ge 2$) as a generalization of the result of $p=0$ in \cite{RBT2016}. Immediately continued in \cite{KY}, we obtained the $p$-Frobenius number for sequences of arithmetic progressions. The $p$-Frobenius numbers of Fibonacci \cite{KY23}, Jacobsthal \cite{KP} and generalized Jacobsthal triples \cite{KLP} have successfully been given just recently.   
The tools from the numerical semigroups we have used so far are still just a few of them. 
By applying the theory introduced and studied in this paper, it is expected to contribute to the linear Diophantine problem of Frobenius and the development to the field of algebraic geometry including the algebraic curve due to the numerical semigroups. That is expected to continue and can be taken into consideration in the near future. 

In this paper, in particular, we would define and study $p$-symmetric semigroups and $p$-pseudo-symmetric semigroups. In order to study numerical semigroups, it is important to study symmetric numerical semigroups. Firstly, every numerical semigroup can be expressed as a finite intersection of irreducible numerical semigroups while irreducible numerical semigroups are either symmetric or pseudo-symmetric. Secondly, the numerical semigroups associated to algebraic curves are free numerical semigroups, hence are symmetric semigroups. Thirdly, the symmetric semigroups correspond to Gorenstein rings. However, we find out that a symmetric semigroup does not keep symmetric in the $p$-case. But, we do find the property in the $p$-case which is similar to symmetry during the research. This motivates us to give the definition of $p$-symmetry. 
We give lots of examples which illustrate how the symmetric semigroups in $0$-case turn into $p$-symmetric semigroups in $p$-case. So, we guess that symmetric semigroups become $p$-symmetric in $p$-case for all symmetric numerical semigroups, hence $p$-symmetric semigroups play the same role in the $p$-case as symmetric semigroups do in the $0$-case.

\section{Preliminaries}

Denote the set of all non-negative integers by $\mathbb N_0$. Let $S$ be a submonoid of $\mathbb N_0$, that is, $S$ satisfies the conditions (i) $S\subseteq\mathbb N_0$; (ii) $0\in S$ (iii) $a,b\in S$ implies $a+b\in S$. Then $S$ is called a {\rm numerical semigroup} if and only if $G(S):=\mathbb N_0\backslash S$ is a finite set. The set $G(S)$ is called the set of {\it gaps} of $S$. The largest integer in $G(S)$ is called the {\it Frobenius number} of $S$ and denoted by $g(S)$.\footnote{Different notation has been used for this and other concepts, depending on the authors and their background. We mainly follow that in \cite{ra05}.} The number $g(S)+1$ is known as the {\it conductor} of $S$. The cardinality of $G(S)$ is called the {\it genus} or the {\it Sylvester number} of $S$ and denoted by $n(S)$. The sum of all the elements (gaps) of $G(S)$ is known as {\it Sylvester sum} of $S$. The numerical semigroup $S$ is called {\it symmetric} if $g(S)-x\in S$ for all $x\in\mathbb Z\backslash S$.      
When $S$ is a numerical semigroup and $A\subseteq S$, it is called that $S$ is generated by $A$ and denoted by $S=\ang{A}$ if for all $n\in S$, there exist $a_1,a_2,\dots,a_k\in A$ and $x_1,x_2,\dots,x_k\in\mathbb N_0$ such that $n=\sum_{j=1}^k a_j x_j$. $A$ is called a {\it minimal set of generators} of $S$ if $S=\ang{A}$ and no proper subset of $A$ has its property. $S=\ang{A}$ is called the {\it canonical form} description of $S$.  
The {\it Ap\'ery set} of $S$ with respect to $n\in S^\ast:=S\backslash\{0\}$, is defined as 
$$ 
{\rm Ape}(S;k)=\{s\in S|s-k\not\in S\}\,.\footnote{In numerical semigroups, the notation ${\rm Ap}(S;k)$ is more popular. However, since the $p$-Ap\'ery set is defined in a later section, it is a little unpleasant to write ${\rm Ap}_p(S;k)$, so this notation is used in this paper.}
$$ 
The Ap\'ery sets were introduced by Ap\'ery in \cite{Apery}. 
For all $j=0,1,\dots,n-1$, let $m_j$ be the least element of $S$ such that $m_j\equiv j\pmod{k}$. Then 
$$
{\rm Ape}(S;k)=\{m_0,m_1,\dots,m_{k-1}\}\,, 
$$ 
that is, ${\rm Ape}(S;k)$ is the complete residue set modulo $n$. 
The {\it Hilbert series} of $S$ is given by 
$$
H(S;x)=\sum_{s\in S}x^s\,. 
$$  
For numerical semigroups, we generally refer to \cite{ADG20,ra05,RG09}. 
In \cite{RR11}, semi-explicit formulas for the Frobenius number, the genus, and the set of pseudo-Frobenius numbers of a certain numerical semigroup $S$ are given. 

Let $a_1,a_2,\dots,a_k$ be positive integers with $\gcd(a_1,a_2,\dots,a_k)=1$. Then, $S=\ang{a_1,a_2,\dots,a_k}$ is a numerical semigroup. The Hilbert series is given by 
$$
H(S;x)=\frac{f(S;x)}{(1-x^{a_1})(1-x^{a_2})\cdots(1-x^{a_k})}\,, 
$$   
where $f(S;x)$ is the polynomial with integral coefficients. 
One of the most important and interesting topics is to find the {\it denumerant} $d(n)=d(n;a_1,a_2,\dots,a_k)$, which is the number of representations to $a_1 x_1+a_2 x_2+\dots+a_k x_k=n$ for a given positive integer $n$. This number is equal to the coefficient of $x^n$ in $1/(1-x^{a_1})(1-x^{a_2})\cdots(1-x^{a_k})$.   
Sylvester \cite{sy1857} and Cayley \cite{cayley} showed that $d(n;a_1,a_2,\dots,a_k)$ can be expressed as the sum of a polynomial of $n$ of degree $k-1$ and a periodic function of period $a_1 a_2\cdots a_k$.  
In \cite{bgk01}, the explicit formula for the polynomial part is derived by using Bernoulli numbers. 
For two variables, a formula for $d(n;a_1,a_2)$ is obtained in \cite{tr00}. For three variables in the pairwise coprime case $d(n;a_1,a_2,a_3)$, in \cite{ko03}, the periodic function part is expressed in terms of trigonometric functions. However, the calculation becomes very complicated for larger $a_1,a_2,a_3$. In \cite{bi20}, three variables case can be easily worked with in his formula using floor functions.   

In this paper, we discuss one of the most natural generalizations in the direction of the number of representations. Let $p$ be a non-negative integer.  
For the set $A=\{a_1,a_2,\dots,a_k\}$, put 
$$
G_p(A)=G_p(\ang{A}):=\{n\in\mathbb N_0|d(n;a_1,a_2,\dots,a_k)\le p\}
$$ 
and 
$$
S_p(A)=S_p(\ang{A})=\{n\in\mathbb N_0|d(n;a_1,a_2,\dots,a_k)>p\}\,,
$$ 
satisfying $G_p(A)+S_p(A)=\mathbb N_0$. We call $S_p(A)$ the {\it $p$-numerical semigroup}.\footnote{This name has often been introduced with different meanings. For example, in \cite{KimKom01}, a numerical semigroup with the least positive integer $n$ is called an $n$-semigroup.} If obvious, we simply write $S_p$ and $G_p$, respectively, When $p=0$, $S=S(A)=S_0(A)$ is a numerical semigroup, which is the set of all the representable non-negative integers in terms of $a_1,a_2,\dots,a_k$, that is, $S=\ang{a_1,a_2,\dots,a_k}$, and $G(A)=G(\ang{A})=G_0(A)$ is the set of gaps of $S$.  
Here, through this paper, unless otherwise specified, $A=\{a_1,a_2,\dots,a_k\}$ is the minimal set of generators with $\gcd(a_1,a_2,\dots,a_k)=1$. 
If there are extra elements in $A$, that is not the minimal set of generators, the result is irrelevant if $p=0$, but the result may change when $p>0$. 
$S_p(A)$ is said to be {\it $p$-generated} from $A$. 
For example, when $A=\{3,10,17\}$, we have  
\begin{align*}
S(A)&=\{3,6,9,10,12,13,15,\mapsto\}\,,\\
S_1(A)&=\{20,23,26,27,29,30,32,\mapsto\}\\
&=\langle{20, 23, 26, 27, 29, 30, \underbrace{32, 33, \dots, 39}\rangle}\,,\\
S_2(A)&=\{30,33,36,37,39,40,42,\mapsto\}\\
&=\langle{30, 33, 36, 37, 39, 40, \underbrace{42, 43, \dots, 59},61,62,64,65,68,71\rangle}\,,\\
S_3(A)&=\{40,43,46,47,49,\mapsto\}\\
&=\langle{40, 43, 46, 47, \underbrace{49, 50, \dots, 79},81,82,84,85,88\rangle}\,,\\
S_4(A)&=\{50,53,54,56,57,59, \mapsto\}\\
&=\langle{50, 53, 54, 56, 57, \underbrace{59, 60, \dots, 99},101,102,105\rangle}\,. 
\end{align*}
where the symbol $\mapsto$ means that all subsequent integers belong to its set.  
Here, the set $S_4(A)$ is $4$-generated from $\{3,10,17\}$, and generated from 
\begin{equation}
\{50, 53, 54, 56, 57, \underbrace{59, 60, \dots, 99},101,102,105\}\,. 
\label{eq:set-s4-3-10-17}
\end{equation} 
This implies that all the elements in $S_4(A)$ are represented in at least $5$ ways in terms of $\{3,10,17\}$ and all the elements in $G_4(A)$ are represented in at most $4$ ways in terms of $\{3,10,17\}$. In addition, all the elements in $S_4(A)$ are represented in at least $1$ way (representable) and all the elements in $G_4(A)$ are not representable in terms of (\ref{eq:set-s4-3-10-17}). 

Notice that strictly speaking, the $p$-numerical semigroup $S_p$ is not a semigroup because for $p>0$, $0\not\in S_p(A)$. $S_p$ is not a monoid, but may be embedded in a monoid formed by adjoining an element $0$, so that for all $s\in S_p^{(0)}(A):=S_p\cup\{0\}$, we have $0+s=s+0=s\in S_p\cup\{0\}$.  

\begin{Lem}  
$S_p^{(0)}(A)$ is a numerical semigroup.   
\label{lem:spanumsg}
\end{Lem} 
\begin{proof} 
It is obvious that $S_p^{(0)}(A)\subseteq\mathbb N_0$ and $0\in S_p^{(0)}(A)$. 
Assume that $\alpha,\beta\in S_p^{(0)}(A)$. Then $\alpha$ and $\beta$ have at least $p+1$ representations in terms of $a_1,\dots,a_k$: 
\begin{align*}  
\alpha&=e_1 a_1+\dots+e_k a_k\,,\\
\beta&=e_{1,1}a_1+\dots+e_{1, k}a_k\\ 
&=e_{2,1}a_1+\dots+e_{2, k}a_k\\ 
\cdots\\
&=e_{p+1,1}a_1+\dots+e_{p+1,k}a_k\,. 
\end{align*} 
Hence, $\alpha+\beta$ has also at least $p+1$ representations: 
\begin{align*}  
\alpha+\beta&=(e_1+e_{1,1})a_1+\dots+(e_k+e_{1, k})a_k\\ 
&=(e_1+e_{2,1})a_1+\dots+(e_k+e_{2, k})a_k\\ 
&\cdots\\
&=(e_1+e_{p+1,1})a_1+\dots+((e_k+e_{p+1,k})a_k\,. 
\end{align*}
Thus, $\alpha+\beta\in S_p^{(0)}(A)$. 
Notice that if $m>p a_1\cdots a_k$, then $m\in S_p^{(0)}(A)$. So, $\#(\mathbb N_0\backslash S_p^{(0)}(A))\le p a_1\cdots a_k$, which is finite.  
Therefore, $S_p^{(0)}(A)$ is a numerical semigroup.  
\end{proof} 

The least element of $S_p(A)$ ($p>0$) or the least nonzero element of $S_p^{(0)}(A)$ is called the $p$-{\it multiplicity} and denoted by $m_p(A):=\min_{n\in S_p(A)}n$.  
Then, $S_p(A)$ is finitely generated and has a minimal generator set (see, e.g., \cite[Proposition 3]{ADG20}). $S_p(A)$ has a minimal generator set $A=\{a_1,\dots,a_k\}$ if and only if all the elements in $S_p(A)$ are represented in at least $p+1$ ways in terms of $a_1,\dots,a_k$ and none of $a_1,\dots,a_k$ can be dropped off. 
The cardinality of a minimal set of generators of $S_p(A)$ is called the $p$-embedding dimension and denoted by $e_p(A):=e_p\bigl(S_p(A)\bigr)$. 
Let $e_p^{(0)}(A)=e(S_p^{(0)}(A))$ and $t_p^{(0)}(A)=t(S_p^{(0)}(A))$. Thus, we have  
$e_p^{(0)}(A)\le m_p^{(0)}(A)$, $t_p^{(0)}(A)\le m_p^{(0)}(A)-1$, $c_p^{(0)}(A)\le(t_p^{(0)}(A)+1)n_p(A)$ and $s_p^{(0)}(A)\le t_p^{(0)}(A)n_p(A)$.

\section{$p$-Frobenius numbers} 

For a non-negative integer $p$, the $p$-Frobenius number and the $p$-Sylvester number ($p$-genus) are denoted by $g_p\bigl(A\bigr)=g\bigl(S_p(A)\bigr)=\max\{n\in G_p(A)\}$ and $n_p\bigl(A\bigr)=n\bigl(S_p(A)\bigr)=\#\{n\in G_p(A)\}$, respectively\footnote{It seems that more authors studying numerical semigroups use the notation $F(A)$ and $g(A)$ to denote the Frobenius number and genus, respectively. We use different notation from different research backgrounds and continuations of previous research.}.  In addition, the $p$-Sylvester sum is denoted by $s_p(A)=\sum_{n\in G_p(A)}n$. When $p=0$, $g(A)=g_0(A)$, $n(A)=n_0(A)$ and $s(A)=s_0(A)$ are the original Frobenius number, Sylvester number (genus) and Sylvester sum, respectively.  

\begin{Lem}
For $p\ge 0$, we have 
$$
n_p(A)\ge\frac{g_p(A)+1}{2}\,. 
$$
\label{lem:n>g12} 
\end{Lem} 
\begin{proof}  
For a non-negative integer $s$, if $s\in S_p$, then $g_p(A)-s\not\in S_p$. Hence, by $n_p(A)\ge\#\{s\in S_p|s<g_p(A)\}=g_p(A)+1-n_p(A)$, we get the result. 
\end{proof}

\begin{Prop}  
Assume that $S_p(A)$ is minimally generated by $a_1,\dots,a_k$. Set $d=\gcd(a_2,\dots,a_{k})$ and $T_p(A)=\{n\in\mathbb N_0|d(n;a_1,a_2/d,\dots,a_k/d)>p\}$ Then we have ${\rm Ape}(S_p,a_1)=d{\rm Ape}(T_p,a_1)$.  
\label{prp:gcd-ape}
\end{Prop} 
\begin{proof}  
From the definition of the Ap\'ery set, $w\in{\rm Ape}(S_p,a_1)$ implies that $w-a_1\not\in S_p$. Since $w\in\ang{a_2,\dots,a_k}$, we have $w/d\in\ang{a_2/d,\dots,a_k/d}$. If $w/d-a_1\in T_p(A)$, as $w-d a_1\not\in S_p$, $w/d-a_1\not\in T_p(A)$. Hence, $w/d\in{\rm Ape}(S_p,a_1)$, which implies that $w\in d{\rm Ape}(S_p,a_1)$. 

On the other hand, if $w\in{\rm Ape}(T_p,a_1)$, then $w\in\ang{a_2/d,\dots,a_k/d}$, implying that $d w\in\ang{a_2,\dots,a_k}\subseteq S_p$. We shall see that $d w-a_1\not\in S_p(A)$, entailing that $d w\in S_p(A)$. Otherwise, for non-negative integers $y_1,\dots,y_k$, $d w-a_1=a_1 y_1+\cdots+a_k y_k$, implying that $w=a_1(y_1+1)/d+(a_2/d)y_2+\cdots+(a_k/d)y_k$ and $d|(y_1+1)$. But this is impossible because $w-a_1\not\in T_p(A)$.  
\end{proof}

By Proposition \ref{prp:gcd-ape}, we can obtain the relations between the $p$-Frobenius numbers $g_p(A)$ and the $p$-Sylvester numbers $n_p(A)$.  For simplicity, we write $g_p\bigl(A_d\bigr)=g\bigl(T_p(A)\bigr)$ and $n_p\bigl(A_d\bigr)=n\bigl(T_p(A)\bigr)$. 

\begin{Cor}  
As the same setting as above, we have 
\begin{enumerate}
\item[\rm (i)] $g_p(A)=d g_p(A_d)+(d-1)a_1$. 
\item[\rm (ii)] $\displaystyle n_p(A)=d n_p(A_d)+\dfrac{(d-1)(a_1-1)}{2}$. 
\item[\rm (iii)] $\displaystyle s_p(A)=d^2 s_p(A_d)+\dfrac{a_1 d(d-1)}{2}n_p(A_d)$$+\dfrac{(a_1-1)(d-1)(2 a_1 d-a_1-d-1)}{2}$. 
\end{enumerate} 
\label{cor:gcd-ape}
\end{Cor}  
\begin{proof}  
We shall prove (iii). By Corollary \ref{cor-mp} (\ref{mp-n}) and (\ref{mp-s}), 
\begin{align*}
s_p(A)&=\frac{1}{2 a_1}\sum_{w\in{\rm Ape}_p(S;a_1)}w^2-\frac{1}{2}\sum_{w\in{\rm Ape}_p(S;a_1)}w+\frac{a_1^2-1}{12}\\
&=\frac{d^2}{2 a_1}\sum_{w\in{\rm Ape}_p(T;a_1)}w^2-\frac{d}{2}\sum_{w\in{\rm Ape}_p(T;a_1)}w+\frac{a_1^2-1}{12}\\
&=d^2\left(\frac{1}{2 a_1}\sum_{w\in{\rm Ape}_p(T;a_1)}w^2-\frac{1}{2}\sum_{w\in{\rm Ape}_p(T;a_1)}w+\frac{a_1^2-1}{12}\right)\\
&\quad -\frac{d}{2}\sum_{w\in{\rm Ape}_p(T;a_1)}w+\frac{a_1^2-1}{12}\\
&=d^2 s_p(A_d)+\dfrac{a_1 d(d-1)}{2}n_p(A_d)+\frac{(a_1-1)(d-1)(2 a_1 d-a_1-d-1)}{2}\,. 
\end{align*}
\end{proof} 

\noindent 
{\bf Example 1.}  
Let $S=\ang{20,30,17}$ and $T=\ang{2,3,17}=\ang{2,3}$ with $d=10$. Then, for $p=3$, by $s_3(A_{10})=136$ and $n_3(A_{10})=17$, we get $s_3(A)=10^2 s_3(A_{10})+17\cdot 10\cdot 9/2 n_3(A_{10})+16\cdot 9(2\cdot 17\cdot 10-17-10-1)/12=30349$.  

\noindent 
{\bf Example 2.}  
When $S=\ang{a,b}$, by putting $d=b$, we get 
$$
g_p(A_d)=a p-1,\quad n_p(A_d)=a p\quad\hbox{and}\quad g_p(A_d)=\frac{a^2 p^2-a p}{2}\,. 
$$ 
Therefore, we have 
\begin{align}  
g_p(a,b)&=(p+1)a b-a-b\,,
\label{eq:g2}\\
n_p(a,b)&=p a b+\frac{(a-1)(b-1)}{2}\,,
\label{eq:n2}\\
s_p(a,b)&=\frac{p^2 a^2 b^2}{2}+\frac{p(a b-a-b)a b}{2}+\frac{(a-1)(b-1)(2 a b-a-b-1)}{12}\,.\label{eq:s2} 
\end{align}

\subsection{$p$-Ap\'ery set} 

For a non-negative integer $p$, 
put ${\rm Ape}_p(A;a)={\rm Ape}_p(\ang{A};a):={\rm Ape}\bigl(S_p(A);a\bigr)$. For all $j=0,1,\dots,n-1$, let $m_j^{(p)}$ satisfy the conditions 
$$
{\rm (i)}\, m_j^{(p)}\equiv j\pmod{a},\quad {\rm (ii)}\, m_j^{(p)}\in S_p(A),\quad {\rm (iii)}\, m_j^{(p)}-a\not\in S_p(A)\,. 
$$ 
Then, the $p$-Ap\'ery set (\cite{ko22}) is given by 
$$
{\rm Ape}_p(A;a)=\{m_0^{(p)},m_1^{(p)},\dots,m_{a-1}^{(p)}\}\,. 
$$ 
Each element $m_j^{(p)}$ of the $p$-Ap\'ery set is useful to obtain the following formula.  

\begin{Prop} 
Let $k$, $p$ and $\mu$ be integers with $k\ge 2$, $p\ge 0$ and $\mu\ge 1$.  
Assume that $\gcd(a_1,a_2,\dots,a_k)=1$ and $a_1=\min(A)$.  We have 
\begin{align*} 
&s_p^{(\mu)}(A):=\sum_{d(n)\le p}n^\mu\\ 
&=\frac{1}{\mu+1}\sum_{\kappa=0}^{\mu}\binom{\mu+1}{\kappa}B_{\kappa}a_1^{\kappa-1}\sum_{i=0}^{a_1-1}\bigl(m_i^{(p)}\bigr)^{\mu+1-\kappa} 
+\frac{B_{\mu+1}}{\mu+1}(a_1^{\mu+1}-1)\,, 
\end{align*} 
where $B_n$ are Bernoulli numbers defined by 
$$
\frac{x}{e^x-1}=\sum_{n=0}^\infty B_n\frac{x^n}{n!}\,. 
$$ 
\label{prp-mp}
\end{Prop} 

When $\mu=0,1$ in Proposition \ref{prp-mp}, together with $g_p(A)$ we have formulae for the $p$-Frobenius number $g_p(A)$, the $p$-Sylvester number $n_p(A)$ and the $p$-Sylvester sum $s_p(A)$.

\begin{Cor}  
Let $k$, $p$ and $\mu$ be integers with $k\ge 2$, $p\ge 0$ and $\mu\ge 1$.  
Assume that $\gcd(a_1,a_2,\dots,a_k)=1$ and $a_1=\min(A)$.  We have 
\begin{align}  
g_p(A)&=\max_{0\le i\le a_1-1}m_i^{(p)}-a_1
\label{mp-g}\,,\\  
n_p(A)&=\frac{1}{a_1}\sum_{i=0}^{a_1-1}m_i^{(p)}-\frac{a_1-1}{2}\,,
\label{mp-n}\\
s_p(A)&=\frac{1}{2 a_1}\sum_{i=0}^{a_1-1}\bigl(m_i^{(p)}\bigr)^2-\frac{1}{2}\sum_{i=0}^{a_1-1}m_i^{(p)}+\frac{a_1^2-1}{12}\,.
\label{mp-s}
\end{align}
\label{cor-mp} 
\end{Cor} 

\noindent 
{\it Remark.}  
When $p=0$, the identities (\ref{mp-g}), (\ref{mp-n}) and (\ref{mp-s}) are essentially due to Brauer and Shockley \cite{bs62}, Selmer \cite{se77} and Tripathi \cite{tr08}\footnote{The formula (\ref{mp-s}) appeared with a typo in \cite{tr08}, and it has been corrected in \cite{pu18}.}, respectively.

\subsection{$p$-Hilbert series}  

For a non-negative integer $p$, the {\it $p$-Hilbert series} of $S_p(A)$ is defined by 
$$
H_p(A;x):=H(S_p;x)=\sum_{s\in S_p(A)}x^s\,. 
$$  
When $p=0$, the $0$-Hilbert series is the original Hilbert series.  
In addition, the $p$-gaps generating function is defined by   
$$
\Psi_p(A;x)=\sum_{s\in G_p(A)}x^s\,,  
$$  
satisfying $H_p(A;x)+\Psi_p(A;x)=1/(1-x)$ ($|x|<1$). 
By using $p$-Ap\'ery set, we see that $S_p(A)={\rm Ape}_p(A;a)+a\mathbb N_0$, with $a=\min(A)$. Hence, 
\begin{equation} 
H_p(A;x)=\frac{1}{1-x^a}\sum_{m\in {\rm Ape}_p(A;a)}x^m\,. 
\label{eq:p-hilbert}
\end{equation}    

For three or more variables, it is not easy to obtain an explicit form of the $p$-Hilbert series. However, the $p$-Hilbert series may be explicitly given when the structure of the $p$-Ap\'ery set is known. 

When $A=\{a,a+d,a+2 d\}$ is the three variable arithmetic sequence\footnote{As mentioned above, in the case of 4 variables or more, the detailed situation has been unknown.} with $a\ge 3$, $d>0$ and $\gcd(a,d)=1$, the structure of the $p$-Ap\'ery set can be found (\cite{KY})\footnote{When $p=0$, the exact structure of a more general sequence $a,a+d,\dots,a+(k-1)d$ is known in \cite{se77}. When $p\ge 1$, it is very hard to know the precise structure of such a general sequence. Hence, we use the case for three variables.}. 
When $a$ is odd, for $0\le p\le\fl{a/2}$, all the elements $(a+d)x_2+(a+2 d)x_3$ of its Ap\'ery set are given as follows. That is, all possible $(x_2,x_3)$ are 
\begin{multline*} 
\left(0,\frac{a-1}{2}+p\right),~\left(1,\frac{a-1}{2}+p-1\right),~\dots,~\\
\left(2 p-2,\frac{a-1}{2}-p+2\right),~\left(2 p-1,\frac{a-1}{2}-p+1\right)
\end{multline*}  
and  
$$
(2 p,0),~\dots,~\left(2 p,\frac{a-1}{2}-p\right),~(2 p+1,0),~\dots,~\left(2 p+1,\frac{a-1}{2}-p-1\right)\,. 
$$  
When $a$ is even, for $0\le p\le\fl{a/2}$, we have all $(x_2,x_3)$ as 
\begin{align*} 
&\left(0,\frac{a}{2}+p-1\right),~\left(1,\frac{a}{2}+p-1\right),\\ 
&\left(2,\frac{a}{2}+p-3\right),~\left(3,\frac{a}{2}+p-3\right),\\ 
&\qquad\cdots\\
&\left(2 p-4,\frac{a}{2}-p+3\right),~\left(2 p-3,\frac{a}{2}-p+3\right),\\ 
&\left(2 p-2,\frac{a}{2}-p+1\right),~\left(2 p-1,\frac{a}{2}-p+1\right)\,, 
\end{align*}  
and   
\begin{align*} 
&(2 p,0),~(2 p,1),~\dots,~\left(2 p,\frac{a}{2}-p-1\right),\\
&(2 p+1,0),~(2 p+1,1),~\dots,~\left(2 p+1,\frac{a}{2}-p-1\right)\,. 
\end{align*}  
Therefore, by (\ref{eq:p-hilbert}), when $a$ is odd, we obtain 
\begin{align*} 
&H_p(a,a+d,a+2 d;x)\\
&=\frac{1}{1-x^a}\biggl(\sum_{i=0}^{2 p-1}x^{i(a+d)+(\frac{a-1}{2}+p-i)(a+2 d)}\\
&\quad +\sum_{i=0}^{\frac{a-1}{2}-p}x^{2 p(a+d)+i(a+2 d)}+\sum_{i=0}^{\frac{a-1}{2}-p-1}x^{(2 p+1)(a+d)+i(a+2 d)}\biggr)\\
&=\frac{1}{1-x^a}\biggl(x^{(\frac{a-1}{2}+p)(a+2 d)}\frac{x^{-2 p d}-1}{x^{-d}-1}\\ 
&\quad +x^{2 p(a+2 d)}\frac{1-x^{(\frac{a-1}{2}-p+1)(a+2 d)}}{1-x^{a+2 d}}
+x^{(2 p+1)(a+2 d)}\frac{1-x^{(\frac{a-1}{2}-p)(a+2 d)}}{1-x^{a+2 d}}\biggr)\\ 
&=\frac{1}{1-x^a}\biggl(\frac{x^{\frac{(a-1)(a+2 d)}{2}+p a+d}(1-x^{2 p d})}{1-x^d}\\
&\qquad +\frac{x^{2 p(a+2 d)}(1+x^{a+2 d}-2 x^{(\frac{a+1}{2}-p)(a+2 d)})}{1-x^{a+2 d}}\biggr)\,.
\end{align*}
When $a$ is even, we obtain 
\begin{align*} 
&H_p(a,a+d,a+2 d;x)\\
&=\frac{1}{1-x^a}\biggl(\sum_{i=0}^{p-1}x^{2 i(a+d)+(\frac{a}{2}+p-1-2 i)(a+2 d)}\\
&\quad +\sum_{i=0}^{p-1}x^{(2 i+1)(a+d)+(\frac{a}{2}+p-1-2 i)(a+2 d)}\\
&\quad +\sum_{i=0}^{\frac{a}{2}-p-1}x^{2 p(a+d)+i(a+2 d)}+\sum_{i=0}^{\frac{a}{2}-p-1}x^{(2 p+1)(a+d)+i(a+2 d)}\biggr)\\
&=\frac{1}{1-x^a}\biggl((x^{a+d}+1)x^{(\frac{a}{2}+p-1)(a+2 d)}\frac{x^{-2 p d}-1}{x^{-2 d}-1}\\ 
&\qquad +(x^{a+d}+1)x^{2 p(a+d)}\frac{1-x^{(\frac{a}{2}-p)(a+2 d)}}{1-x^{a+2 d}}\biggr)\\ 
&=\frac{x^{a+d}+1}{1-x^a}\biggl(\frac{x^{\frac{a(a+2 d)}{2}+(p+1)a}(1-x^{2 p d})}{1-x^{2 d}}\\
&\qquad +\frac{x^{2 p(a+d)}(1-x^{(\frac{a}{2}-p)(a+2 d)})}{1-x^{a+2 d}}\biggr)\,.
\end{align*}

\begin{Prop} 
Let $p$, $a$ and $d$ be integers with $p\ge 0$, $a\ge 3$, $d>0$ and $\gcd(a,d)=1$. Then 
for $0\le p\le\fl{a/2}$, we have 
\begin{align*}
&H_p(a,a+d,a+2 d;x)\\
&=
\left\{
\begin{alignedat}{2}
& \frac{1}{1-x^a}\biggl(\frac{x^{\frac{(a-1)(a+2 d)}{2}+p a+d}(1-x^{2 p d})}{1-x^d} && \\
&\qquad +\frac{x^{2 p(a+2 d)}(1+x^{a+2 d}-2 x^{(\frac{a+1}{2}-p)(a+2 d)})}{1-x^{a+2 d}}\biggr)&\quad&\text{if $a$ is odd},\\
& \frac{x^{a+d}+1}{1-x^a}\biggl(\frac{x^{\frac{a(a+2 d)}{2}+(p+1)a}(1-x^{2 p d})}{1-x^{2 d}} && \\
&\qquad +\frac{x^{2 p(a+d)}(1-x^{(\frac{a}{2}-p)(a+2 d)})}{1-x^{a+2 d}}\biggr)&\quad&\text{if $a$ is even}\,.
\end{alignedat}
\right. 
\end{align*} 
\label{prp:p-hilbert-arithm}
\end{Prop}

\section{$p$-symmetric semigroup}

For $p\ge 0$, assume that the $p$-Ap\'ery set of $A=\{a_1,\dots,a_k\}$ with $a_1=\min(A)$ is given by 
\begin{align*}
{\rm Ape}_p(A;a_1)&=\{m_0^{(p)},m_1^{(p)},\dots,m_{a_1-1}^{(p)}\}\\
&=\{\ell_0(p),\ell_1(p),\dots,\ell_{a_1-1}(p)\}\,, 
\end{align*}
with $\ell_0(p)<\ell_1(p)<\dots<\ell_{a_1-1}(p)$. That is, the sequence $\ell_0(p)$,$\ell_1(p)$,$\dots$,$\ell_{a_1-1}(p)$ is the ascending permutation of $m_0^{(p)},m_1^{(p)},\dots,m_{a_1-1}^{(p)}$.   

The $p$-numerical semigroup $S_p$, which is $p$-generated from $A$, is called {\it $p$-symmetric} if for all $x\in\mathbb Z\backslash S_p$, $\ell_0(p)+g_p(A)-x\in S_p$, where $\ell_0(p)$ is the least element of $S_p$, that is the {\it $p$-multiplicity} of $S_p$ if $p\ge 1$; $\ell_0(p)=0$ if $p=0$. 
When $p=0$, "$0$-symmetric" is "symmetric".  
If a $p$-symmetric numerical semigroup $S_p$ further satisfies $\ell_0(p)=g_p(A)+1:=c_p(A)$, which is called $p$-conductor, then $S_p$ is called {\it $p$-completely-symmetric}.    

From the definition, the following conditions are obvious. 

\begin{Prop} 
For a $p$-semigroup $S_p$ ($p\ge 0$), the following conditions are equivalent.  
\begin{enumerate}
\item[{\rm (i)}] $S_p$ is $p$-symmetric. 
\item[{\rm (ii)}] $\displaystyle \# S_p\cap\{\ell_0(p),\dots,g_p(A)\}=\# G_p\cap\{\ell_0(p),\dots,g_p(A)\}=\frac{g_p(A)-\ell_0(p)+1}{2}$. 
\item[{\rm (iii)}] If $x+y=\ell_0(p)+g_p(A)$, then exactly one of non-negative integers $x$ and $y$ belongs to $S_p$ and another to $G_p$. 
\end{enumerate}
\label{prp:p-sym}
\end{Prop}

\noindent 
{\bf Examples}  



When $A=\{3,10,17\}$, $S_p$ is $p$-symmetric for $p=1,2,7,8,9,10,11,19,20$, $\dots$, $37,38,39,40$. In addition, $S_p$ is $p$-completely-symmetric for $p=19,20$, $\dots$, $36,39,40$. For example, $G_{19}(A)=\{0,1,\dots,126\}$ and $S_{19}(A)=\{127,128,\mapsto\}$, so $S_{19}$ is $19$-completely-symmetric.  
$S_p$ is not $p$-symmetric but $p$-pseudo-symmetric (see below) for $p=0,3,4,5,6,12,13,14,15,16,17,18$.

 
When $A=\{4,5,6\}$, we get that 
\begin{align*}  
S_8&=\{36,38,40,41,\mapsto\}\,,\\
G_8&=\{0,1,\dots,35,37,39\}\,. 
\end{align*} 
Then we know that 
$$
g_8(A)=39\quad\hbox{and}\quad {\rm Ape}_8(A)=\{36=m_0^{(8)},38=m_2^{(8)},41=m_1^{(8)},43=m_3^{(8)}\}\,. 
$$ 
Hence, we see that 
\begin{align*} 
x&\in G_8\cup\mathbb Z^{-}=\{\leftarrowtail,-2,-1,0,1,\dots,35,37,39\}\\
&\Longleftrightarrow\\
36+39-x&\in S_8=\{36,38,40,41,\mapsto\}\,. 
\end{align*} 
Therefore, $S_8$, which is $8$-generated from $A=\{4,5,6\}$, is $8$-symmetric.  
In fact, we can obtain 
$$
36+43=38+41\,. 
$$ 
This fact is a generalization of the result by Ap\'ery \cite{Apery}. 

\begin{Lem}
For a non-negative integer $p$, $S_p$, which is $p$-generated from $A$, is $p$-symmetric if and only if $\ell_i(p)+\ell_{a-i-1}(p)=g_p(A)+\ell_0(p)+a$ ($i=1,2,\dots,\fl{a/2}$). 
\label{lem:sym-apery}
\end{Lem}

If one element $m_i^{(p)}$ in ${\rm Ape}_p(A;a)$ with $a=\min(A)$ can extend such that $m_i^{(p)}\equiv i\pmod a$ for any $i$, Lemma \ref{lem:sym-apery} can be restated as follows. For simplicity, put $g=g_p(A)$ and $\ell=\ell_0(p)$. 

\begin{Lem}
For a non-negative integer $p$, $S_p$, which is $p$-generated from $A$, is $p$-symmetric if and only if $m_{(g+\ell+1)/2+j}(p)+m_{(g+\ell-1)/2+j}(p)=g_p+\ell+a$ ($j\in\mathbb Z$). 
\label{lem:sym-apery-m}
\end{Lem}

From Lemma \ref{lem:sym-apery} or Lemma \ref{lem:sym-apery-m}, we have a relation between $p$-Frobenius number $g_p(A)$ and $p$-Sylvester number $n_p(A)$. 


\begin{Prop}  
For a non-negative integer $p$, $S_p$, which is $p$-generated from $A$, is $p$-symmetric if and only if 
$$
n_p(A)=\frac{g_p(A)+\ell_0(p)+1}{2}\,. 
$$ 
\label{prp:sym-n-g}
\end{Prop}  
\begin{proof}  
By the formula (\ref{mp-n}), we have 
\begin{align*}  
2 n_p(A)&=\frac{2}{a}\sum_{i=0}^{a-1}m_i^{(p)}-a+1\\
&=\frac{1}{a}\sum_{i=0}^{a-1}\bigl(m_{(g+\ell+1)/2+i}^{(p)}+m_{(g+\ell-1)/2+i}^{(p)}\bigr)-a+1\\
&=\ell+g_p+1\,. 
\end{align*} 
On the other hand, by (\ref{mp-n}) again, we have 
\begin{align*}  
&g_p(A)+\ell_0(p)+1=2 n_p(A)\\
&=\frac{1}{a}\sum_{i=0}^{a-1}\bigl(m_{(g+\ell+1)/2+i}^{(p)}+m_{(g+\ell-1)/2+i}^{(p)}\bigr)-a+1\\
&=\frac{1}{a}\sum_{i=0}^{a-1}\bigl(g_p(A)+\ell_0(p)+a c_i\bigr)-a+1\\
&=g_p(A)+\ell_0(p)-a+1+\frac{1}{a}\sum_{i=0}^{a-1}c_i\,. 
\end{align*} 
Since $\sum_{i=0}^{a-1}c_i=a$, we know that $c_i=1$ for all $i$. 
\end{proof}


Let us consider the two variables' case.  
For any integer $n\in S_p(A)$ for $A=\{a,b\}$ with $\gcd(a,b)=1$ and $a<b$, let $x_0$ be the largest integer $x$ satisfying $n=a x+b y$ ($y\ge 0$). Then there exists the least non-negative integer $y_0$ such that $n=a x_0+b y_0$, which is called the {\it standard form} of the representation of $n$. Since $S_p(A)\subseteq\mathbb N\subseteq\mathbb Z$ and $\mathbb Z$ is Euclidean domain, the standard form is unique. 

\begin{Lem}  
Let $n=a x_0+b y_0$ be the standard form of $n$. Then 
\begin{enumerate} 
\item[\rm (i)] $0\le y_0\le a-1$. 
\item[\rm (ii)] For any integer $n\in S(A)=S_0(A)$, $n\in S_p(A)$ if and only if $x_0\ge p b$. 
\end{enumerate}
\label{lem:standard-form}
\end{Lem} 
\begin{proof}  
(i) is obvious.  \\ \noindent  
(ii) Assume that $n=a x_0+b y_0$ and $x_0<p b$. Since $x_0-p b<0$, $n=a x_0+b y_0=a(x_0-b)+b(y_0+a)=\cdots=a\bigl(x_0-(p-1)b\bigr)+b\bigl(y_0+(p-1)a\bigr)$ has at most $p$ representations. But $n\in S_p(A)$ implies that $n$ must have more than $p$ representations. This is the contradiction.  
On the contrary, if $x_0\ge pb$, then $n=a x_0+b y_0$ has at least $p+1$ representations, so $n\in S_p(A)$. 
\end{proof}

By Proposition \ref{prp:sym-n-g} and Lemma \ref{lem:standard-form}, together with the formulas (\ref{eq:g2}) and (\ref{eq:n2}), we can show the $p$-symmetric property for two variables.

\begin{theorem}  
For any non-negative integer $p$, $S_p(a,b)$ with $\gcd(a,b)=1$ is $p$-symmetric.  
\label{th:2-psym}
\end{theorem}
\begin{proof} 
When $A=\{a,b\}$ with $\gcd(a,b)=1$, by Lemma \ref{lem:standard-form}, the least integer whose number of representations in terms of $a$ and $b$ is more than $p$ is $p a b$. That is, the non-negative integral solutions of $a x+b y=p a b$ are $(x,y)=\bigl(j b, (p-j)a\bigr)$ ($j=0,1,\dots,p$). Since $\ell_0(p)=p a b$, by Proposition \ref{prp:sym-n-g} together with the formulas (\ref{eq:g2}) and (\ref{eq:n2}), we have 
\begin{align*} 
\frac{g_p(A)+\ell_0(p)+1}{2}&=\frac{(p+1)a b-a-b+p a b+1}{2}\\
&=p a b+\frac{(a-1)(b-1)}{2}=n_p(a,b)\,. 
\end{align*}
\end{proof}

For three or more variables, when $p=0$, several results have been known (e.g., \cite{EL94, GRJ17,Matt04}). However, when $p\ge 1$, it is not easy to decide whether a given $p$-numerical semigroup is $p$-symmetric. Here, we show a result for arithmetic sequence.    

\begin{theorem}  
Let $a$ and $d$ be integers with $a\ge 3$ even, $d>0$ and $\gcd(a,d)=1$. Then $S_p(a,a+d,a+2 d)$ is $p$-symmetric when $p=a/2-1$.  
\label{th:3arithm-psym}
\end{theorem} 
\begin{proof}  
In \cite{KY}, it is shown that for $0\le p\le\fl{a/2}$, 
\begin{align*}
&g_p(a,a+d,a+2 d)=(a+2 d)p+\fl{\frac{a-2}{2}}a+(a-1)d\,,\\   
&n_p(a,a+d,a+2 d)\\
&=\begin{cases}
(2 a+2 d-1-p)p+\frac{(a-1)(a+2 d-1)}{4}&\text{if $a$ is odd};\\ 
(2 a+2 d-1-p)p+\frac{(a-1)(a+2 d-1)+1}{4}&\text{if $a$ is even}\,.
\end{cases} 
\end{align*} 
In addition, from the proof of Proposition \ref{prp:p-hilbert-arithm}, we can know that 
the least element in $S_p(a,a+d,a+2d)$ is given by $\ell_0(p)=2 p(a+d)$ when $p$ is odd or $p$ is even and $p\le(a/2-1)(a+2 d)/a$, by comparing another candidate $(a/2+p-1)(a+2 d)$. 
By using Proposition \ref{prp:sym-n-g},  when $a$ is odd, we have 
$$
(a-2-2 p)p=\frac{1}{2}\,, 
$$ 
which is impossible for integers $a$ and $p$.  When $a$ is even, we have $(a-2-2 p)p=0$. Hence, $p=a/2-1$, satisfying $p\le(a/2-1)(a+2 d)/a$.  
\end{proof}

\noindent 
{\it Remark.}  
When $a$ is odd or $a$ is even and $p>a/2$, any general result has not been known. 
\bigskip

At the end of this section, we consider a $p$-symmetric property in terms of the valuation. Let $R_0:=\mathbb{K}[[t^s|s\in S_p^{(0)}(A)]]$, $\bar{R}_0$ be the integral closure of $R_0$, $f$ be the algebraic conductor from $t^{\ell_0(p)}R_0$ to $\bar{R}_0$, $c_p=g_p+1$ ($p$-conductor). Since $R_0$ is the ring associated to a numerical semigroup $S_p^{(0)}(A)$, 
it is a discrete valuation ring with the valuation $v$. 
	
\begin{Lem} 
$f=\{x\in \bar{R}_0|v(x)\ge c_p+\ell_0(p)\}$. 
\label{lem:valuation}
\end{Lem} 
\begin{proof} 
For any $x\in f$ and $r\in \bar{R}_0$, we have $r x\in t^{\ell_0(p)}R_0$. So $v(r x)=v(x)+v(r)\in v(t^{\ell_0(p)}R_0)$. For $x=t^{\ell_0(p)}x'$ we get $v(r)+v(t^{\ell_0(p)})+v(x')\in v(t^{\ell_0(p)}R_0)=v(R_0)+v(t^{\ell_0(p)})$.  
By the arbitrariness of $r$ and $v(r)\ge 0$, we have $v(x)\ge c_p+\ell_0(p)$, that is, $f\subseteq \{x\in \bar{R}_0|v(x)\ge c_p+\ell_0(p)\}$. 

For any $x\in \bar{R}_0$ and $v(x)\ge c_p+\ell_0(p)$, we have $v(x)=v(r)$ for some $r\in t^{\ell_0(p)}R\subseteq R$. Then for any $r'\in \bar{R}_0$, we have $v(x r')=v(x)+v(r')=v(r)+v(r')\ge c_p+\ell_0(p)$. By the definition of $c_p$, we have $x r'\in t^{\ell_0(p)}R_0$. So, $f\supseteq \{x\in \bar{R}_0|v(x)\ge c_p+\ell_0(p)\}$. 
\end{proof} 

For simplicity, let $d_1$ and $d_2$ be the lengths of ideal of $R_0/f$ and of $R_0$-submodule of $\bar{R}_0/f$, respectively, and $d_3$ be the number of elements in $S_p(A)\cap\{1,2,\dots c_p+\ell_0(p)-1\}$. 

\begin{theorem}
$S_p(A)$ is $p$-symmetric if and only if $d_1=\frac{d_2}{2}$. 
\label{th:valuation}
\end{theorem}
\begin{proof}
By Proposition \ref{prp:p-sym} together with the facts that all the elements in $\{1,\dots i_p-1\}$ are in $G_p(A)$ and $\{g_p+1,\dots \ell_0(p)+g_p-1\}$ are all in $S_p(A)$, $S_p(A)$ is $p$-symmetric if and only if $d_3=\frac{\ell_0(p)+g_p-1}{2}$. 

Consider the ideal chain $R_0\supset R_1\supset R_2\dots \supset R_{d_3}\supset f$, where $R_i=\{r\in R_0|v(r)\ge v_i\}$ and $v_1<v_2<\dots <v_{d_3}$ are the elements in $S_p(A)\cap\{1,2,\dots c_p+\ell_0(p)-1\}$ arranged in ascending order. This sequence is maximal because if we adjoin an element $r\in R_0$ of value $v_{i-1}$ to $R_i$, we get all of $R_{i-1}$. So, $d_1=d_3+1$. \\
Similarly consider the maximal $R_0$-submodule chain of $\bar{R}_0/f$: $\bar{R}_0=b_0\supset b_1\supset b_2\dots \supset b_{\ell_0(p)+g_p+1}=f$ where $b_i=\{r\in \bar{R}_0|v(r)\ge i\}$. So we have $d_2=\ell_0(p)+g_p+1$. 
Hence, $S_p(A)$ is $p$-symmetric if and only if $d_1-1=\frac{d_2-2}{2}$.
\end{proof}

\subsection{$p$-pseudo-symmetric semigroup} 

For a non-negative integer $p$, let $S_p(A)$ be a $p$-numerical semigroup. $x\in\mathbb Z$ is called a {\it $p$-pseudo-Frobenius number} if $x\not\in S_p(A)$ and $x+s-\ell_0(p)\in S_p(A)$ for all $s\in S_p(A)\backslash\{\ell_0(p)\}$, where $\ell_0(p)$ is the least element of $S_p(A)$, so is of ${\rm Ape}_p(A;a)$ with $a=\min(A)$. The set of $p$-pseudo-Frobenius numbers is denoted by ${\rm PF}_p(A)$. The {\it $p$-type} is denoted by $t_p(A):=\#\bigl({\rm PF}_p(A)\bigr)$. Notice that the $p$-Frobenius number is given by $g_p(A)=\max\bigl({\rm PF}_p(A)\bigr)$. 

\noindent 
{\bf Example.} 
When $A=\{6,17,28\}$ and $p=5$, 
\begin{multline*}
S_5(A)=\{130=\ell_0(5),136,142,147,148,152,153,154,158,159,160,\\
164,165,166,\underbrace{168,\dots,172},\underbrace{174,\dots,178},180,\mapsto\}\,, 
\end{multline*} 
so, 
\begin{multline*}
\bigl(S_5(A)-\ell_0(5)\bigr)\backslash\{0\}=\{6,12,17,18,22,23,24,28,29,30,\\
34,35,36,\underbrace{38,\dots,42},\underbrace{44,\dots,48},50,\mapsto\}\,.
\end{multline*} 
Hence, ${\rm PF}_5(6,17,28)=\{163,179\}$, and $t_5(6,17,28)=2$.

For $p\ge 0$, the $p$-numerical semigroup $S_p$, which is $p$-generated from $A$, is called {\it $p$-pseudo-symmetric} if for all $x\in\mathbb Z\backslash S_p$ with $x\ne\bigl(\ell_0(p)+g_p(A)\bigr)/2\in\mathbb Z$, $\ell_0(p)+g_p(A)-x\in S_p$, where $\ell_0(p)$ is the least element of $S_p$. 
When $p=0$, "$0$-pseudo-symmetry" is "pseudo-symmetry".  

\noindent 
{\bf Example.} 
When $A=\{3,7,11\}$, $S_p(A)$ is $p$-pseudo-symmetric for $p=0,3,4,5$ and $p$-symmetric for $p=1,2$. For example, by $G_4=\{0,\dots,34,37\}$ and $S_4=\{35,36,38,\mapsto\}$, for any $x\in S_4$, $72-x\in G_4\cup\mathbb Z^{-}$ except the element $36=72/2$. By $G_5=\{0,\dots,38,40,41\}$ and $S_5=\{39,42,43,\mapsto\}$, for any $x\in S_5$, $80-x\in G_5\cup\mathbb Z^{-}$ except the element $40=80/2$.

For simplicity, put the $p$-Frobenius number as $g:=g_p(A)$ and the $p$-multiplicity as $\ell:=\ell_0(p)$ ($p\ge 1$) with $\ell_0(0)=0$. Denote the $p$-Ap\'ery set by ${\rm Ape}_p(A;a)$ with $a=\min(A)$. 

\begin{theorem} 
For a non-negative integer $p$, the following conditions are equivalent:  
\begin{enumerate} 
\item[{\rm (i)}] $S_p$, which is $p$-generated from $A$, is $p$-pseudo-symmetric 
\item[{\rm (ii)}] 
$$ 
m_{(g+\ell)/2+j}^{(p)}+m_{(g+\ell)/2-j}^{(p)}=g+\ell+
\begin{cases}
2 a&\text{if $j=0$ and $(g+\ell)/2\in G_p(A)$};\\ 
0&\text{if $j=0$ and $(g+\ell)/2\in S_p(A)$};\\ 
a&\text{if $j>0$}. 
\end{cases}
$$ 
\item[{\rm (iii)}] $\displaystyle n_{p}(A)=\dfrac{g+\ell}{2}+\begin{cases}
1&\text{if $(g+\ell)/2\in G_p(A)$};\\ 
0&\text{if $(g+\ell)/2\in S_p(A)$}. 
\end{cases}$
\end{enumerate}
\label{th:ps-sym-ap}
\end{theorem} 

\begin{proof} 
\noindent 
[(i)$\Rightarrow$(ii)]  
When $j=0$, by $m_{(g+\ell)/2}^{(p)}\equiv(g+\ell)/2\pmod{a}$, the result is clear. 
Notice that 
\begin{align*}
m_{(g+\ell)/2+j}^{(p)}+m_{(g+\ell)/2-j}^{(p)}&\equiv (g+\ell)/2+j+(g+\ell)/2-j\\
&=g+\ell\pmod{a}\,.
\end{align*} 
When $j>0$, suppose that $m_{(g+\ell)/2+j}^{(p)}+m_{(g+\ell)/2-j}^{(p)}\le g+\ell$. Then there exists a non-negative integer $\rho$ such that $m_{(g+\ell)/2+j}^{(p)}+m_{(g+\ell)/2-j}^{(p)}=g+\ell-\rho a$. So, $m_{(g+\ell)/2+j}^{(p)}+\bigl(m_{(g+\ell)/2-j}^{(p)}+\rho a\bigr)=g+\ell$ with $m_{(g+\ell)/2+j}^{(p)},m_{(g+\ell)/2-j}^{(p)}+\rho a\in S_p(A)$ contradicts the assumption (i).  
Thus, by $m_{(g+\ell)/2+j}^{(p)}+m_{(g+\ell)/2-j}^{(p)}>g+\ell$, there exists a positive integer $\rho$ such that $m_{(g+\ell)/2+j}^{(p)}+m_{(g+\ell)/2-j}^{(p)}=g+\ell+\rho a$. 
As $m_{(g+\ell)/2-j}^{(p)}-a\not\in S_p(A)$, by (i) we get $m_{(g+\ell)/2+j}^{(p)}-(\rho-1)a=g+\ell-\bigl(m_{(g+\ell)/2-j}^{(p)}-a\bigr)\in S_p(A)$. Hence, only the possibility is that $\rho=1$.  
 
\noindent 
[(ii)$\Rightarrow$(i)]  
Suppose that $S_p(A)$ is not $p$-pseudo-symmetric. Then there exists an integer $\mu$ such that $\mu\in S_p(A)$ and $\mu\ne g+\ell-\mu\in S_p(A)$ or $\mu\not\in S_p(A)$ and $\mu\ne g+\ell-\mu\not\in S_p(A)$. 
If $\mu\in S_p(A)$ and $\mu\ne g+\ell-\mu\in S_p(A)$, by $m_{(g+\ell)/2+j}^{(p)}+m_{(g+\ell)/2-j}^{(p)}\equiv g+\ell\pmod{a}$, there exists a non-negative integer $j_0$ such that 
$$
m_{(g+\ell)/2+j_0}^{(p)}\equiv\mu\quad\hbox{and}\quad m_{(g+\ell)/2-j_0}^{(p)}\equiv g+\ell-\mu\pmod{a}\,. 
$$ 
And there exist non-negative integers $h_1$ and $h_2$ such that 
\begin{align*}
g+\ell&=\mu+(g+\ell-\mu)\\
&=(m_{(g+\ell)/2+j_0}^{(p)}+h_1 a)+(m_{(g+\ell)/2-j_0}^{(p)}+h_2 a)\\
&=(h_1+h_2)a+g+\ell+\begin{cases}
2 a&\text{if $j_0=0$ and $(g+\ell)/2\in G_p(A)$};\\ 
0&\text{if $j_0=0$ and $(g+\ell)/2\in S_p(A)$};\\ 
a&\text{if $j_0>0$}. 
\end{cases}
\end{align*}
By the assumption of (ii), only the possibility is that $h_1=h_2=0$ and $j_0=0$ with $(g+\ell)/2\in S_p(A)$. Thus, $m_{(g+\ell)/2}^{(p)}=\mu=g+\ell-\mu$,  
yielding the contradiction. 

If $\mu\in G_p(A)$ and $\mu\ne g+\ell-\mu\in G_p(A)$, there exists a non-negative integer $j_1$ and there exist positive integers $h_3$ and $h_4$ such that 
\begin{align*}
g+\ell&=\mu+(g+\ell-\mu)\\
&=(m_{(g+\ell)/2+j_1}^{(p)}-h_3 a)+(m_{(g+\ell)/2-j_1}^{(p)}-h_4 a)\\
&=-(h_3+h_4)a+g+\ell+\begin{cases}
2 a&\text{if $j_1=0$ and $(g+\ell)/2\in G_p(A)$};\\ 
0&\text{if $j_1=0$ and $(g+\ell)/2\in S_p(A)$};\\ 
a&\text{if $j_1>0$}. 
\end{cases}
\end{align*}
By the assumption of (ii), only the possibility is that $h_3=h_4=1$ and $j_1=0$ with $(g+\ell)/2\in G_p(A)$. Thus, $m_{(g+\ell)/2}^{(p)}=\mu=g+\ell-\mu$, 
yielding the contradiction. 

\noindent 
[(ii)$\Leftrightarrow$(iii)] 
By (\ref{mp-n}), we have 
\begin{align*}
2 n_p(A)&=\frac{2}{a}\sum_{i=0}^{a-1}m_i^{(p)}-a+1\\
&=\frac{1}{a}\sum_{j=0}^{a-1}\bigl(m_{(g+\ell)/2+j}^{(p)}+m_{(g+\ell)/2-j}^{(p)}\bigr)-a+1\\ 
&=g+\ell-a+1+
\begin{cases}
a+1&\text{if $(g+\ell)/2\in G_p(A)$}\\ 
a-1&\text{if $(g+\ell)/2\in S_p(A)$} 
\end{cases}\\
&=g+\ell+
\begin{cases}
2&\text{if $(g+\ell)/2\in G_p(A)$}\\ 
0&\text{if $(g+\ell)/2\in S_p(A)$} 
\end{cases}
\,. 
\end{align*} 
On the other hand, by (\ref{mp-n}) again, we have 
\begin{align*}
&g+\ell+
\begin{cases}
2&\text{if $(g+\ell)/2\in G_p(A)$}\\ 
0&\text{if $(g+\ell)/2\in S_p(A)$}
\end{cases}\\ 
&=2 n_p(A)=\frac{1}{a}\sum_{j=0}^{a-1}\bigl(m_{(g+\ell)/2+j}^{(p)}+m_{(g+\ell)/2-j}^{(p)}\bigr)-a+1\\ 
&=\frac{1}{a}\sum_{i=0}^{a-1}(g+\ell+a c_i)-a+1\\ 
&=g+\ell-a+1+\frac{1}{a}\sum_{i=0}^{a-1}c_i\,. 
\end{align*} 
Thus, 
$$
\sum_{i=0}^{a-1}c_i=a+
\begin{cases}
1&\text{if $(g+\ell)/2\in G_p(A)$}\\ 
-1&\text{if $(g+\ell)/2\in S_p(A)$}\,. 
\end{cases}
$$ 
Hence, $c_i=1$ ($i\ne 0$) and 
$$
c_0=\begin{cases}
2&\text{if $(g+\ell)/2\in G_p(A)$}\\ 
0&\text{if $(g+\ell)/2\in S_p(A)$}\,. 
\end{cases}
$$ 
\end{proof}

When $A=\{6,7,17,28\}$, $S_p$ is not $p$-symmetric for $p=0,1,\dots,11,13,14,16$ and is $p$-symmetric for $p=15$. $S_{12}$ is $12$-pseudo-symmetric, because $G_{12}=\{0,\dots,85,88\}$ and $S_{12}=\{86,87,89,\mapsto\}$, so ${\rm Ape}_{12}=\{m_2^{(12)}=86,m_3^{(12)}=87,m_5^{(12)}=89,m_0^{(12)}=90,m_1^{(12)}=91,m_4^{(12)}=94\}$. Hence, by $m_{g+\ell}^{(12)}=m_{174}^{(12)}=m_{0}^{(12)}=90$ and $m_{(g+\ell)/2}^{(12)}=m_{87}^{(12)}=m_{3}^{(12)}=87$, we have $m_3^{(12)}+m_3^{(12)}=174=g+\ell$, $m_4^{(12)}+m_32^{(12)}=m_5^{(12)}+m_1^{(12)}=m_0^{(12)}+m_0^{(12)}=180=g+\ell+a$, where $g=88$, $\ell=86$ and $a=6$. 
By $(g+\ell)/2=87\in S_p$, we have 
$$
n_{12}(A)=\frac{g+\ell}{2}=87\,. 
$$ 

$S_{17}$ is $17$-pseudo-symmetric, because $G_{17}=\{0,\dots,96,99,100,101\}$ and $S_{17}=\{97,98,102,\mapsto\}$, so ${\rm Ape}_{17}=\{m_1^{(17)}=97,m_2^{(17)}=98,m_0^{(17)}=102,m_3^{(17)}=105,m_4^{(17)}=106,m_5^{(17)}=107\}$. Hence, by $m_{g+\ell}^{(17)}=m_{198}^{(17)}=m_{0}^{(17)}=102$ and $m_{(g+\ell)/2}^{(17)}=m_{99}^{(17)}=m_{3}^{(17)}=105$, we have $m_3^{(17)}+m_3^{(17)}=210=g+\ell+2 a$, $m_4^{(17)}+m_2^{(17)}=m_5^{(17)}+m_1^{(17)}=m_0^{(17)}+m_0^{(17)}=204=g+\ell+a$, where $g=101$, $\ell=97$ and $a=6$. 
By $(g+\ell)/2=87\in G_p$, we have 
$$
n_{17}(6,7,17,28)=\frac{g+\ell}{2}+1=99+1=100\,. 
$$

\begin{Cor}  
Let $S_p(A)$ be a $p$-numerical semigroup. The following conditions are equivalent.  
\begin{enumerate} 
\item[\rm (i)] $S_p$ is $p$-symmetric. 
\item[\rm (ii)] ${\rm PF}_p(A)=\{g_p(A)\}$ with $g_p(A)\not\equiv \ell_0(p)\pmod 2$. 
\item[\rm (iii)] $t_p(A)=1$ with $g_p(A)\not\equiv \ell_0(p)\pmod 2$. 
\end{enumerate} 
\label{cor:p-sym} 
\end{Cor}

\begin{Cor}  
Let $S_p(A)$ be a $p$-numerical semigroup. The following conditions are equivalent.  
\begin{enumerate} 
\item[\rm (i)] $S_p$ is $p$-pseudo-symmetric. 
\item[\rm (ii)] $\displaystyle {\rm PF}_p(A)=
\begin{cases}
\{g_p(A),\bigl(g_p(A)+\ell_0(p)\bigr)/2\}&\text{if $\bigl(g_p(A)+\ell_0(p)\bigr)/2\in G_p(A)$};\\
\{g_p(A)\}&\text{if $\bigl(g_p(A)+\ell_0(p)\bigr)/2\in S_p(A)$}. 
\end{cases}$  
\item[\rm (iii)] $\displaystyle t_p(A)=
\begin{cases}
2&\text{if $\bigl(g_p(A)+\ell_0(p)\bigr)/2\in G_p(A)$};\\
1&\text{if $\bigl(g_p(A)+\ell_0(p)\bigr)/2\in S_p(A)$}. 
\end{cases}$ 
\end{enumerate} 
\label{cor:p-pseudo-sym} 
\end{Cor}

For $a,b\in\mathbb Z$, define a partial order relation $a\le_{S_p} b$ (or $a\le_{S} b$ for short) as $b-a\in S_p$. The set of $p$-pseudo-Frobenius numbers ${\rm PF}_p(A)$ can be determined with this order relation in terms of the $p$-maximal gaps.

\begin{Prop}  
Let $S_p(A)$ be a $p$-numerical semigroup, which is $p$-generated from $A$. Then we have 
$$
{\rm PF}_p(A)={\rm Maximals}_{\le_S}(G_p)\,. 
$$ 
\label{prp:pf-max} 
\end{Prop}
\begin{proof} 
\noindent  
[$x\in {\rm LHS}\Longrightarrow x\in{\rm RHS}$] 
If $x\in{\rm PF}_p(A)\subset S_p$, then $x\not\in S_p(A)$ and $x+s-\ell_0(p)\in S_p(A)$ for all $s\in S_p(A)\backslash\{\ell_0(p)\}$. If $x\not\in{\rm Maximals}_{\le_S}(G_p)$, then there exists $y\in G_p$ such that $x\le_S y$. 
If $x\ne y$, then $s:=y-x\in S_p$, so by $\ell_0(p)\in S_p$, we get $y=x+s\in S_p$, which is a contradiction. 

\noindent 
[$x\in{\rm RHS}\Longrightarrow x\in {\rm LHS}$] 
If $x\in{\rm Maximals}_{\le_S}(G_p)$ and $x\not\in{\rm PF}_p(A)$, then $x\not\in S_p$ and there exists $s\in S_p\backslash\{\ell_0(p)\}$ such that $x+s-\ell_0(p)\not\in S_p$. As $\ell_0(p)\in S_p$, $x+s\not\in S_p$. So, $x\le_S x+s$ with $s=(x+s)-x\in S_p$, yielding a contradiction to the maximality of $x$.    
\end{proof}

The set of $p$-pseudo-Frobenius numbers ${\rm PF}_p(A)$ can be also determined in terms of the $p$-Ap\'ery set. 

\begin{Prop}  
Let $S_p(A)$ be a $p$-numerical semigroup, which is $p$-generated from $A$ with $a=\min(A)$. Then for $n\in S_p$ we have 
$$
{\rm PF}_p(A)=\{w-a|w\in{\rm Maximals}_{\le_S}{\rm Ape}_p(A;a)\}\,. 
$$ 
\label{prp:pf-ape} 
\end{Prop}
\begin{proof} 
\noindent 
[$x\in {\rm LHS}\Longrightarrow x\in{\rm RHS}$] 
If $x\in{\rm PF}_p(A)$, then $x\not\in S_p$ and $x+a\in S_p$, where $a+\ell_0(p):=s\in S_p\backslash\{\ell_0(p)\}$. Thus, $x+a\in{\rm Ape}_p(A;a)$. 
If $x+a$ is not maximal with respect to $\le_S$, there exists $w\in{\rm Ape}_p(A;a)$ such that $x+a\le_S w$. So, $w-a\not\in S_p$ and $s:=w-x-a\in S_p$.  
However, by $x\in{\rm PF}_p(A)$, $x\not\in S_p$ and $x+s-\ell_0(p)\in S_p$ for all $s\in S_p\backslash\{\ell_0(p)\}$. As $\ell_0(p)\in S_p$, $x+s\in S_p$, which is a contradiction.

\noindent 
[$x\in{\rm RHS}\Longrightarrow x\in {\rm LHS}$] 
If $w\in{\rm Ape}_p(A;a)$, then $w-a\not\in S_p$. 
If $w-a\not\in{\rm PF}_p(A)$, then $w-a\not\in S_p$ and there exists $s\in S_p\backslash\{\ell_0(p)\}$ such that $w-a+s-\ell_0(p)\not\in S_p$. 
As $\ell_0(p)\in S_p$, $w+s-a\not\in S_p$. 
As $w,s\in S_p$, $w+s\in S_p$. Hence, $w+s\in{\rm Ape}_p(A;a)$.   
Because $(w+s)-w=s\in S_p$, $w\le_S w+s$, which is a contradiction to the maximality of $x$.    
\end{proof}

\noindent 
{\bf Example.} 
When $A=\{6,17,28\}$ and $p=5$, we saw that ${\rm PF}_5(6,17,28)=\{163,179\}$, and $t_5(6,17,28)=2$. Hence, by Corollary \ref{cor:p-sym}, $S_5(A)$ is not $5$-symmetric though $g_p(A)=179\not\equiv \ell_0(p)=130\pmod 2$. Since $\bigl(g_p(A)+\ell_0(p)\bigr)/2\not\in\mathbb Z$, by Corollary \ref{cor:p-pseudo-sym}, $S_5(A)$ is not $5$-pseudo-symmetric either. 

Since 
\begin{multline*}
G_5(A)=\{\underbrace{0,1,\dots,129},\underbrace{131,\dots,135},\underbrace{137,\dots,141},\underbrace{143,\dots,146},\\
149,150,151,155,156,157,161,162,163,167,173,179\}\,,
\end{multline*}
the largest elements of each residue in $G_5(A)$ are $126\equiv 4$, $141\equiv 3$, $146\equiv 2$, $162\equiv 0$, $163\equiv 1$, $179\equiv 5\pmod 6$, respectively. However, $126,141,146,162$ are not maximal because $126+17=143$, $141+22=163$, $146+17=163$ and $162+17=179$ belong to $G_5(A)$ with $17,22\in S_5(A)-\ell_0(5)$. Therefore, by Proposition \ref{prp:pf-max}, we have ${\rm PF}_5(6,17,28)=\{163,179\}$.  

Since ${\rm Ape}_5(A)=\{130=m_4^{(5)},147=m_3^{(5)},152=m_2^{(5)},168=m_0^{(5)},169=m_1^{(5)},185=m_5^{(5)}\}$, we get ${\rm Ape}_5(A)-130=\{0,17,22,38,39,55\}$. 
Since $17+22=39$ and $38+17=55$, we see that $17,22,38$ are not maximal with respect to $\le_S$. Therefore, by Proposition \ref{prp:pf-ape}, we have ${\rm PF}_5(6,17,28)=\{169-6,185-6\}=\{163,179\}$.

At the end of this subsection, we mention a partially corresponding result to Theorem \ref{th:valuation}.  

\begin{theorem}
If $S_p(A)$ is $p$-pseudo-symmetric, then $2d_1+1=d_2$.
\end{theorem} 
\begin{proof}
If $S_p(A)$ is $p$-pseudo-symmetric, then we have $2d_3=\ell_0(p)+g_p-2$.\\
Again, consider the maximal ideal chain $R_0\supset R_1\supset R_2\dots \supset R_{d_3}\supset f$ as in the proof of Theorem \ref{th:valuation}. Thus, we get $d_1=d_3+1$. And consider the $R_0$-submodule chain of $\bar{R}_0/f$: $\bar{R}_0=b_0\supset b_1\supset b_2\dots \supset b_{\ell_0(p)+g_p+1}=f$. We have $d_2=\ell_0(p)+g_p+1$. 
Hence, if $S_p(A)$ is $p$-pseudo-symmetric, then $2(d_1-1)=d_2-3$. 
\end{proof}

\subsection{$p$-irreducible numerical semigroup}  

A numerical semigroup $S$ is irreducible if it cannot be expressed as the intersection of two proper oversemigroups.  
A $p$-numerical semigroup $S_p$, which is $p$-generated from $A$, is called {\it $p$-irreducible} if it is either $p$-symmetric or $p$-pseudo-symmetric.  It is known that every numerical semigroup can be expressed as a finite intersection of irreducible numerical semigroups.  

By Theorem \ref{th:2-psym}, we have the $p$-irreducible property for two variables.  

\begin{Cor}  
For any non-negative integer $p$, $S_p(a,b)$ with $\gcd(a,b)=1$ is $p$-irreducible.  
\label{cor:2-psym}
\end{Cor}

Every $p$-numerical semigroup can be also expressed as a finite intersection of irreducible numerical semigroups (\cite{ADG20}).  

\begin{Prop}  
For a non-negative integer $p$, let $S_p$ be a $p$-numerical semigroup. Then, there exist finitely many irreducible numerical semigroups $\mathcal S_1,\dots,\mathcal S_r$ such that $S_p=\mathcal S_1\cup\dots\cup\mathcal S_r$. 
\label{prp:irredu-decomp} 
\end{Prop} 

\noindent 
{\it Remark.}  
It has not been known that for any fixed non-negative integer $p$, a $p$-numerical semigroup can be expressed as an intersection of $p$-irreducible numerical semigroups.

\noindent 
{\bf Example.} 
For $A=\{5,9,16\}$, we see that $S_2(A)=\{41,45,46,48,50,\mapsto\}$, which is neither ($2$-)symmetric nor ($2$-)pseudo-symmetric. But it can be expressed as an intersection of two ($0-$)numerical semigroups: $S_2(A)=S(A_1)\cup\mathcal S(A_2)$ with $A_1=\{41,43,45,46,48,50,\mapsto\}$ and $A_2=\{41,45,46,47,48,50,\mapsto\}$. Here both $S(A_1)$ and $S(A_2)$ are ($0-$)pseudo-symmetric. In addition, these $0$-numerical semigroups are given by canonical forms: 
\begin{align*}  
S(A_1)&=\langle{41,43,45,46,48,\underbrace{50,\dots,81},83,85\rangle}\,,\\
S(A_2)&=\langle{41,45,46,47,48,\underbrace{50,\dots,81},83,84,85\rangle}\,. 
\end{align*}

\section*{Statements and Declarations} 

The authors have no conflicts of interest directly relevant to the content of this article. Funding information is not available. 




\begin{thebibliography}{99}

\bibitem{ADL16}  
I. Aliev, J. A. de Loera and Q. Louveaux, {\em 
Parametric polyhedra with at least $k$ lattice points: their semigroup structure and the $k$-Frobenius problem}, 
Beveridge, Andrew (ed.) et al., Recent trends in combinatorics. Cham: Springer. The IMA Volumes in Mathematics and its Applications 159, 753--778 (2016). 

\bibitem{Apery}  
R. Ap\'ery,  {\em 
Sur les branches superlin\'eaires des courbes alg\'ebriques}, 
C. R. Acad. Sci. Paris {\bf 222} (1946), 1198--1200. 

\bibitem{ADG20} 
A. Assi, M. D'Anna and P. A. Garcia-Sanchez, {\em 
Numerical semigroups and applications}, Second edition, 
RSME Springer Series, 3. Springer, Cham, 2020.  


\bibitem{bgk01}   
M. Beck, I. M. Gessel and T. Komatsu, {\em    
The polynomial part of a restricted partition function related to the Frobenius problem}, 
Electron. J. Combin. {\bf 8} (No.1) (2001), \#N7.

\bibitem{bi20} 
D. S. Binner, {\em 
The number of solutions to $a x+b y+c z=n$ and its relation to quadratic residues}, 
J. Integer Seq. {\bf 23}, No. 6, (2020), Article 20.6.5, 19 p.  

\bibitem{bs62} 
A. Brauer and B. M. Shockley, {\em  
On a problem of Frobenius}, 
J. Reine. Angew. Math. {\bf 211} (1962), 215--220.  

\bibitem{BDFHKMRSS} 
A. Brown, E. Dannenberg, J. Fox, J. Hanna, K. Keck, A. Moore, Z. Robbins, B. Samples and J. Stankewicz, {\em 
On a generalization of the Frobenius number}, 
arXiv:1001.0207 (2010). 

\bibitem{cayley} 
A. Cayley, {\em 
On a problem of double partitions}, 
Philos. Mag. {\bf XX} (1860), 337--341.

\bibitem{cu90} 
F. Curtis, {\em 
On formulas for the Frobenius number of a numerical semigroup}, 
Math. Scand. {\bf 67} (1990), 190--192. 


\bibitem{EL94} 
M. Estrada and A. Lopez, {\em 
A note on symmetric semigroups and almost arithmetic sequences}, 
Commun. Algebra {\bf 22} (1994), No. 10 3903--3905. 

\bibitem{FS11} 
L. Fukshansky and A. Schurmann, {\em 
Bounds on generalized Frobenius numbers}, 
Eur. J. Comb. {\bf 32} (2011), No. 3, 361--368.  

\bibitem{GRJ17} 
I. Garc\'ia-Marco, J. L. Ram\'irez Alfonsin and \O. J. R\o dseth, {\em 
Numerical semigroups. II: Pseudo-symmetric AA-semigroups}, 
J. Algebra {\bf 470} (2017), 484--498. 

\bibitem{KimKom01}
S. J. Kim and J. Komeda, {\em 
Numerical semigroups which cannot be realized as semigroups of Galois Weierstrass points}, 
Arch. Math. (Basel) {\bf 76} (2001), No. 4, 265--273. 

\bibitem{ko03}  
T. Komatsu, {\em  
On the number of solutions of the Diophantine equation of Frobenius--General case}, 
Math. Commun. {\bf 8} (2003), 195--206. 

\bibitem{ko22}  
T. Komatsu, {\em  
On $p$-Frobenius and related numbers due to $p$-Ap\'ery set}, 
arXiv:2111.11021v4 (2023).  

\bibitem{Ko-repunit}  
T. Komatsu, {\em  
The Frobenius number associated with the number of representations for sequences of repunits}, 
C. R. Acad. Sci., Paris, Ser. I, Math. {\bf 361} (2023), 73--89.    

\bibitem{Ko-triangular} 
T. Komatsu, {\em 
The Frobenius number for sequences of triangular numbers associated with number of solutions},   
Ann. Comb. {\bf 26} (2022) 757--779. 

\bibitem{KLP}  
T. Komatsu, S. Laishram and P. Punyani, {\em  
$p$-numerical semigroups of generalized Fibonacci triples}, 
Symmetry {\bf 15} (2023), no.4, Article 852, 13 p. 
https://doi.org/10.3390/sym15040852 

\bibitem{KP}  
T. Komatsu and C. Pita-Ruiz, {\em  
The Frobenius number for Jacobsthal triples associated with number of solutions}, 
Axioms {\bf 12} (2023), no.2, Article 98, 18 p.   
https://doi.org/10.3390/axioms12020098 

\bibitem{KY}  
T. Komatsu and H. Ying, {\em  
The Frobenius number for sequences of arithmetic progressions associated with the number of solutions}, 
arXiv:2206.13052 (2022).   

\bibitem{KY23}  
T. Komatsu and H. Ying, {\em  
The $p$-Frobenius and $p$-Sylvester numbers for Fibonacci and Lucas triplets}, 
Math. Biosci. Eng. {\bf 20} (2023), No.2, 3455--3481.

\bibitem{Matt04}
G. L. Matthews, {\em 
On numerical semigroups generated by generalized arithmetic sequences}, 
Commun. Algebra {\bf 32} (2004), No. 9, 3459--3469. 

\bibitem{pu18}  
P. Punyani and A. Tripathi, {\em  
On changes in the Frobenius and Sylvester numbers},  
Integers {\bf 18B} (2018), \#A8, 12 p. 

\bibitem{ra05} 
J. L. Ram\'irez Alfons\'in,  {\em 
The Diophantine Frobenius Problem},  
Oxford University Press, Oxford, 2005.  

\bibitem{RR11} 
A. M. Robles-P\'erez and J. C. Rosales, {\em 
The Frobenius problem for numerical semigroups with embedding dimension equal to three}, 
Math. Comput. {\bf 81} (2012), No. 279, 1609--1617. 

\bibitem{RR18}  
A. M. Robles-P\'erez and J. C. Rosales,  {\em  
The Frobenius number for sequences of triangular and tetrahedral numbers}, 
J. Number Theory {\bf 186} (2018), 473--492.  

\bibitem{RBT2015}  
J. C. Rosales, M. B. Branco and D. Torr\~ao, {\em 
The Frobenius problem for Thabit numerical semigroups},  
J. Number Theory {\bf 155} (2015), 85--99.  

\bibitem{RBT2016}  
J. C. Rosales, M. B. Branco and D. Torr\~ao, {\em 
The Frobenius problem for repunit numerical semigroups},  
Ramanujan J. {\bf 40} (2016), 323--334.  

\bibitem{RBT2017}  
J. C. Rosales, M. B. Branco and D. Torr\~ao, {\em 
The Frobenius problem for Mersenne numerical semigroups},  
Math. Z. {\bf 286} (2017), 741--749.  

\bibitem{RG09} 
J. C. Rosales and P. A. Garcia-Sanchez, {\em 
Numerical semigroups}, 
Developments in Mathematics, 20. Springer, New York, 2009.  
 
\bibitem{se77}  
E. S. Selmer, {\em  
On the linear diophantine problem of Frobenius},  
J. Reine Angew. Math. {\bf 293/294} (1977), 1--17.  

\bibitem{sy1857} 
J. J. Sylvester, {\em 
On the partition of numbers}, 
Quart. J. Pure Appl. Math. {\bf 1} (1857), 141--152.

\bibitem{tr00}  
A. Tripathi, {\em 
The number of solutions to $a x+b y=n$}, 
Fibonacci Quart. {\bf 38} (2000), 290--293.

\bibitem{tr08}  
A. Tripathi, {\em 
On sums of positive integers that are not of the form $a x+b y$}, 
Amer. Math. Monthly {\bf 115} (2008), 363--364.  


\end{thebibliography}
\end{document}